\documentclass[a4paper,10pt]{article}
\usepackage{mathtext}
\usepackage[T1,T2A]{fontenc}
\usepackage[cp1251]{inputenc}
\usepackage[english]{babel}
\usepackage{amsmath}
\usepackage{amsfonts}
\usepackage{amssymb}
\usepackage{mathrsfs}
\usepackage{amsthm}
\usepackage{euscript}

\newcommand{\eps}{\varepsilon}

\DeclareMathOperator{\BV}{BV}
\DeclareMathOperator{\Id}{Id}
\DeclareMathOperator{\Var}{Var}
\DeclareMathOperator{\curl}{curl}
\DeclareMathOperator{\BMO}{BMO}
\DeclareMathOperator{\Div}{div}

\renewcommand{\leq}{\leqslant}
\renewcommand{\geq}{\geqslant}

\theoremstyle{plain}\newtheorem{Th}{Theorem}
\theoremstyle{plain}\newtheorem{Le}{Lemma}
\theoremstyle{plain}\newtheorem{Lemma}[Le]{Lemma}
\theoremstyle{plain}
\theoremstyle{plain}\newtheorem{Con}{Conjecture}
\theoremstyle{plain}

\begin{document}

\title{Dimension of gradient measures\thanks{The authors are greatful to A.A.~Logunov for suggestions concerning exposition of the material.}}
\author{D.~M.~Stolyarov\thanks{Supported by: Chebyshev Laboratory (SPbU), RF Government grant no. 11.G34.31.0026; JSC ``Gazprom Neft''; RFBR grant no. 14-01-00198 A.} \and M.~Wojciechowski\thanks{Supported by NCN grant no. N N201 607840.}}

\maketitle
\begin{abstract}
We prove that if pure derivatives with respect to all coordinates of a function on~$\mathbb{R}^n$ are signed measures\textup, then their lower Hausdorff dimension is at least~$n-1$. The derivatives with respect to different coordinates may be of different order. 
\end{abstract}

\section{Introduction}\label{s1}
We begin with a well-known fact: if a function~$f$ is in~$\BV$, then the lower Hausdorff dimension of~$\nabla f$ is not less than~$n-1$ (see~\cite{AFP}, Lemma~$3.76$). By the lower Hausdorff dimension of a vector-valued complex measure~$\mu$ we mean 
\begin{equation}\label{DimensionOfMeasure}
\dim \mu = \inf \{\alpha\mid \exists F\hbox{~--- Borel set},\quad~\mu(F) \ne 0,\quad \dim F \leq \alpha\}.
\end{equation}
In~\cite{RW}, this fact was treated as a manifestation of some more general uncertainty-type principle. We use notation from that paper. Namely, let~$\phi: S^{n-1} \to S^{n-1}$ be a mapping. Consider the class~$M_{\phi}$ of vector-valued  signed measures~$\mu$ such that~$\hat{\mu}(\xi) \parallel \phi(\frac{\xi}{|\xi|})$ for all~$\xi \in \mathbb{R}^n \setminus\{0\}$. From the celebrated theorem of Uchiyama~\cite{U}, it follows that if~$\phi(\xi)$ is not parallel to~$\phi(-\xi)$ for all~$\xi \in S^{n-1}$, then every~$\mu$ from~$M_{\phi}$ is absolutely continuous. However, can one say something if this condition does not hold? We cite a simpler version Theorem~$3$ from~\cite{RW}.
\begin{Th}\label{GeneralPhi}
Suppose that the image of~$\phi$ contains~$n$ linearly independent points~$\phi(h_1),\phi(h_2),\ldots,\phi(h_n)$ and~$\phi$ is~$\alpha$-H\"older in neighborhoods of~$h_i$\textup,~$i = 1,2,\ldots,n$,~$\alpha > \frac12$. Then~$\dim \mu \geq 1$ for all~$\mu \in M_{\phi}$.
\end{Th}
The relationsip between~$\BV$ and~$M_{\phi}$ can be expressed by the formula~$\{\nabla f \mid f \in \BV(\mathbb{R}^n)\} = M_{i\Id}$, where~$i\Id$ is the identity map on the sphere. In this particular case, Theorem~\ref{GeneralPhi} is weaker (we get only dimension~$1$). 
One can make a courageous conjecture (Conjecture~$1$ in~\cite{RW}).
\begin{Con}\label{GeneralConj}
Suppose that the function~$\phi$ is Lipschitz and its image contains~$n$ linearly independent points. Then~$\dim \mu \geq n-1$ for all~$\mu \in M_{\phi}$.
\end{Con}
Not being able to prove the conjecture, we state a result that lies towards it. In what follows,~$D_i$ means ``the derivative with respect to~$x_i$''. Unless specified, all measures are not assumed to be positive or real-valued.
\begin{Th}\label{DerivativesTheorem}
Let~$m$ be a natural number. Let~$f$ be a function such that for all~$i$~$D_i^m f$ is a measure. Then~$\dim \mu \geq n-1$\textup, where~$\mu$ is a vector-valued measure whose components are~$D_i^m f$.
\end{Th}
This theorem is a particular case of Conjecture~\ref{GeneralConj},~$\mu \in M_{\phi}$, where
\begin{equation*}
\phi(\xi) = \frac{(i\xi)^m}{|\xi^m|}.
\end{equation*}
When the orders of derivation differ, the homogeneity is not isotropic. However, in this case we still have the same principle.
\begin{Th}\label{DifferentDerivativesTheorem}
Let~$m_1,m_2,\ldots,m_n$ be natural numbers. Let~$f$ be a function such that for all~$i$~$D_i^{m_i} f$ is a measure. Then~$\dim \mu \geq n-1$\textup, where~$\mu$ is a vector-valued measure whose components are~$D_i^{m_i} f$.
\end{Th}

The basic fact about~$\BV$-functions we statrted with can be proved by several methods. In~\cite{AFP}, the proof is based on the coarea formula for~$\BV$-functions. It gives more information about those ``parts'' of~$\nabla f$ that have dimension~$n-1$: they are situated on the jumps of~$f$. However, the applicability of the methods from~\cite{AFP} both to Conjecture~\ref{GeneralConj} and even to Theorem~\ref{DerivativesTheorem} is questionable. The proof of Theorem~\ref{GeneralPhi} is based on application of the classical~F.~and~M. Rieszs' theorem (see~\cite{K}, p.~$28$) on the continuity of analytic signed measure. This gives only dimension~$1$ (however, allows one to avoid all the algebraic structure of~$\phi$). 

Our strategy is, in some sense, a mixture of the two proofs indicated above. The coarea formula is replaced by the Sobolev embedding theorem with the limiting summation exponent, and the Rieszs' theorem is replaced by some modification of the Frostman lemma. 

In Section~\ref{s2} we prove Theorem~\ref{DifferentDerivativesTheorem} (and Theorem~\ref{DerivativesTheorem} as a particular case), except for the modification of the Frostman lemma, which we prove in~Section~\ref{s3}. The last Section~\ref{s4} contains some examples and some suggestions how to prove Conjecture~\ref{GeneralConj}.

\section{Proof of Theorem~\ref{DifferentDerivativesTheorem}}\label{s2}
We begin with an exposition of the embedding theorem we are going use. We need some Besov spaces. The reader unfamiliar with them can either consult~\cite{BIN,P}, or skip the details and pass to Theorem~\ref{NaturalEmbedding} directly. 

By~$W_1^{m}$,~$m = (m_1,m_2,\ldots,m_n)$, we denote the completion of the set~$C_0^{\infty}(\mathbb{R}^n)$ in the norm
\begin{equation}\label{SobolevNorm}
\|f\|_{W_1^{m}} = \sum\limits_{i=1}^n \big\|D_i^{m_i} f\big\|_{L^1}.
\end{equation}
Another norm on the set~$C_0^{\infty}(\mathbb{R}^n)$ describes the one-dimensional Besov spaces (i.e. we measure the norm of a single derivative of a function in~$\mathbb{R}^n$),
\begin{equation}\label{BesovNorm}
\|f\|_{B^{i,\ell}_{q,\theta}} = \bigg(\int\limits_0^{\infty} \Big(h^{-\ell}\big\| \Delta_i^s(h) f\big\|_{L^q(\mathbb{R}^n)}\Big)^{\theta} \frac{dh}{h}\bigg)^{\frac{1}{\theta}}.
\end{equation}
Here~$i$ is the number of the coordinate,~$i = 1,2,\ldots,n$;~$\ell$ is the order of differentiation,~$0 < \ell$;~$s$ is an auxiliary integer parameter,~$\ell < s$;~$\Delta_i^{s}(h)$ is the finite difference operator of order~$s$ and step~$h$ with respect to the~$i$-th coordinate;~$\theta$ is the interpolation parameter.

We cite Theorem~$4$ from~\cite{Kol} (see Theorem~$\mathrm{B}$ in~\cite{Kol2} and~\cite{Kol3} also).
\begin{Th}\label{BesovEmbedding}
Let~$f$ be a function in~$W_1^{m}$. Then\textup, for each~$i = 1,2,\ldots,n$ and any~$\ell_i < m_i$\textup, the inequality
\begin{equation*}
\|f\|_{B^{i,\ell_i}_{q,1}} \lesssim \sum_{j=1}^n\big \|D_j^{m_j} f\big\|_{L^1}
\end{equation*}
holds true\textup, if the parameters satisfy the homogeneity condition
\begin{equation*}
\ell_i = m_i\Big(1 - \frac{q-1}{q}\sum\limits_{j=1}^n\frac{1}{m_j}\Big).
\end{equation*} 
\end{Th}
Now we fix~$\ell_i = m_i - 1$, therefore,
\begin{equation}\label{Homq}
\frac{q-1}{q} = \Big(\sum\limits_{j=1}^n \frac{m_i}{m_j}\Big)^{-1}.
\end{equation}
In particular, if all the derivatives are of equal order,~$q = \frac{n}{n-1}$. The equality~\eqref{Homq} matches its individual~$q$ to each~$m_i$, we denote it by~$q_i$.
Using the easy embedding (see~\cite{P}, p. 62) for~$\theta = 1$
\begin{equation*}
\big\| D_i^{m_i - 1}f\big\|_{L^{q_i}} \lesssim \|f\|_{B^{i,m_i-1}_{q_i,1}},
\end{equation*}
we get the following embedding theorem without Besov spaces.
\begin{Th}\label{NaturalEmbedding}
Let~$f$ be a function in~$W_1^{m}$. Then\textup, for each~$i = 1,2,\ldots,n$\textup,
\begin{equation*}
\big\|D_i^{m_i - 1} f\big\|_{L^{q_i}} \lesssim \sum_{j=1}^n \big \|D_j^{m_j} f\big\|_{L^1},
\end{equation*}
if the parameters satisfy the homogeneity condition~\eqref{Homq}.
\end{Th}
Suppose now that~$f$ is a function with compact support such that~$\mu_i = D_i^{m_i} f$ is a measure for all~$i$. Then, 
\begin{equation}\label{MeasureEmbedding}
\big\|D_i^{m_i - 1}f\big\|_{L^{q_i}} \lesssim \sum_{j=1}^n \Var \mu_j.
\end{equation}
This can be deduced from Theorem~\ref{NaturalEmbedding} by a simple limiting argument. We skip the details.

Let~$\varphi$ be a function in~$C_0^{\infty}(\mathbb{R}^{n-1})$ supported in a unit ball. Let~$\varphi_r(x) = \varphi(r^{-1}x)$,~$r > 0$. For~$x = (x_1,x_2,\ldots,x_n) \in \mathbb{R}^n$ we write~$x_{[i]}$ for a~$(n-1)$-dimensional vector that is obtained from~$x$ by forgetting the~$i$-th coordinate (for example, for~$n=3$,~$x_{[2]} = (x_1,x_3)$). By~$B_{r}(z)$ we denote a~$(n-1)$-dimensional ball of radius~$r$ centered at~$z$
\begin{Le}\label{MainEstimate}
Let the balls~$B_{r_j}(y_j)$ be disjoint\textup, let~$\psi \in C_0^{\infty}(\mathbb{R})$ be a test function. Suppose that~$f$ is a compactly supported function. If~$\mu = (D_i^{m_i} f)_i $ is a measure\textup, then\textup, for all~$i = 1,2,\ldots,n$ and any~$\varphi \in C^{\infty}_0$ supported in a unit ball\textup,
\begin{equation*}
\sum\limits_j\Big|\int\limits_{\mathbb{R}^n}\psi(x_i)\varphi_{r_j}(x_{[i]} + y_j)\,d\mu_i(x)\Big| \lesssim \Big(\sum\limits_{j}r_j^{n-1}\Big)^{\frac{1}{q'_i}}\Var\mu
\end{equation*}
for some fixed~$q'_i$ and all~$y_j \in\mathbb{R}^{n-1}$ and~$r_j < 1$ uniformly \textup(the constants may depend on~$\varphi$ and~$\psi$\textup).
\end{Le}
\begin{proof}
For simplicity, let~$i=1$. We can write
\begin{multline*}
\sum\limits_j\Big|\int\limits_{\mathbb{R}^n}\psi(x_1)\varphi_{r_j}(x_{[1]} + y_j)\,d\mu_1(x)\Big| = \sum\limits_j\Big|\int\limits_{\mathbb{R}^n}\psi'(x_1)\varphi_{r_j}(x_{[1]} + y_j)D_1^{m_1-1}f(x)\,dx\Big| \leq \\
\sum\limits_j\Big\|\psi'(x_1)\varphi_{r_j}(x_{[1]} + y_j)\Big\|_{L^{q'_1}}\Big\|D_1^{m_1-1}f\Big\|_{L^{q_1}\big(B_{r_j}(y_j)\big)} \lesssim \\ \sum\limits_j r_j^{\frac{n-1}{q_1'}}\Big\|D_1^{m_1-1}f\Big\|_{L^{q_1}\big(B_{r_j}(y_j)\big)} \leq \Big(\sum\limits_j r_j^{n-1}\Big)^{\frac{1}{q'_1}}\Big(\sum\limits_j \Big\|D_1^{m_1-1}f\Big\|_{L^{q_1}\big(B_{r_j}(y_j)\big)}^{q_1}\Big)^{\frac{1}{q_1}} \leq \\
\Big(\sum\limits_{j}r_j^{n-1}\Big)^{\frac{1}{q_1'}}\big\|D_1^{m-1}f\big\|_{L^{q_1}} \lesssim
\Big(\sum\limits_{j}r_j^{n-1}\Big)^{\frac{1}{q'_1}}\|\mu\|
\end{multline*}
Here~$q_1$ is the one taken from Theorem~\ref{MeasureEmbedding}, and~$q'_1$ is its adjoint. The first inequality is the H\"older inequality, the second one is rescaling, the third one is H\"older again, and the fourth one is inequality~\eqref{MeasureEmbedding}.
\end{proof}
The next lemma is a generalization of Frostman's lemma (see~\cite{M}, p. 112, for the original). 
\begin{Le}\label{FrostmanGeneralization}
Suppose that~$\varphi \in C^{\infty}_0(\mathbb{R}^n)$ is a radial non-negative function supported in a unit ball that decreases as the radius grows\textup,~$\varphi(x) = 1$ when~$|x| \leq\frac34$.  
Let~$\mu$ be a measure such that for every collection~$B_{r_j}(x_j)$ of~$n$-dimensional balls such that~$B_{3r_j}(x_j)$ are disjoint the estimate
\begin{equation*}
\sum\limits_{j}\Big|\int\limits_{\mathbb{R}^n}\varphi_{3r_j}(x_j + y)\,d\mu(y)\Big| \lesssim \big(\sum\limits r_j^{\alpha}\big)^{\beta}
\end{equation*}
holds true for some positive~$\alpha$ and~$\beta$. Then~$\dim(\mu) \geq \alpha$.
\end{Le}
We will prove it in Section~\ref{s3}. 
\begin{Lemma}\label{MeasureSection}
Let~$\mu$ be a Borel measure on~$\mathbb{R}^{k+l}$. Suppose that~$\mu(I\times A) = 0$ for every parallelepiped~$I \subset \mathbb{R}^k$ and every Borel~$A\subset\mathbb{R}^l$ such that~$\dim A < \alpha$. Then~$\dim \mu \geq \alpha$.
\end{Lemma}
\begin{proof}
First, we prove that~$\mu \big|_{I\times A} = 0$ for every~$I$ and~$A$ as above. Indeed, the~$\sigma$-algebra of measurable subsets of~$I\times A$ is generated by the sets~$\tilde{I}\times (A\cap J)$, where~$J$ is an arbitrary parallelepiped in~$\mathbb{R}^l$ and~$\tilde{I}$ is a parallelepiped inside~$I$. By the assumptions,~$\mu(\tilde{I}\times \big(A\cap J)\big) = 0$ (because~$\dim A\cap J \leq \dim A <\alpha$). Therefore,~$\mu \big|_{I\times A} = 0$.

Assume the contrary, suppose that~$\mu(F) \ne 0$, but~$\dim F < \alpha$. Then, $\dim \pi_{\mathbb{R}^l}[F] < \alpha$, because the projection does not increase the dimension (it is a Lipshitz mapping). So,~$\mu \big|_{\mathbb{R}^k \times\pi_{\mathbb{R}^l}[F]} = 0$. But~$F \subset \mathbb{R}^k\times \pi_{\mathbb{R}^l}[F]$, which contradicts~$\mu(F) \ne 0$.
\end{proof}

\paragraph{Proof of Theorem~\ref{DifferentDerivativesTheorem}.}
Assume the contrary, let~$F$ be some Borel set such that~$\mu(F)\ne 0$, but~$\dim F < n-1$. We may assume that~$\mu_1(F) \ne 0$ (by symmetry) and~$F$ is compact (due to the regularity of measure). Multiplying~$f$ by a test function that equals~$1$ on~$F$, we make~$f$ compactly supported without loosing the condition that its higher order derivatives are signed measures. To get a contradiction, it is sufficient to prove that for every set~$A \subset \mathbb{R}^{n-1}$ such that~$\dim A < n-1$ and every function~$\psi \in C_0^{\infty}(\mathbb{R})$
\begin{equation}\label{ZeroFormula}
\int\limits_{\mathbb{R}\times A} \psi(x_1)\,d\mu_1(x) = 0.
\end{equation}
Then, approximating the characteristic functions of intervals by smooth functions, we get the hypothesis of Lemma~\ref{MeasureSection} with~$\alpha = n-1$, which, in its turn, asserts that~$\mu_1(F)=0$.

Consider now a signed measure~$\mu_{\psi}$ on~$\mathbb{R}^{n-1}$ given by formula
\begin{equation*}
\mu_{\psi}(B) = \int\limits_{\mathbb{R}\times B}\psi(x_1)\,d\mu_1(x).
\end{equation*}
By Lemma~\ref{MainEstimate},~$\mu_{\psi}$ satisfies the hypothesis of Lemma~\ref{FrostmanGeneralization} with~$\alpha = n-1$,~$k=1$,~$l=n-1$. Therefore,~$\dim\mu_{\psi} \geq n-1$ and equality~\eqref{ZeroFormula} holds for all~$A$ with~$\dim A < n-1$.

\section{Proof of Lemma~\ref{FrostmanGeneralization}}\label{s3}
To prove Lemma~\ref{FrostmanGeneralization}, we need some preparation. Obviously, if~$\mu$ is a complex-valued measure, then 
\begin{equation*}
\dim\mu = \min(\dim\Re\mu,\dim\Im\mu).
\end{equation*}
Therefore, it is enough to prove Lemma~\ref{FrostmanGeneralization} for real-valued signed measures only. 

Next lemma provides a softer substitute for the Lebesque differentiation theorem for an arbitrary Borel measure. 
\begin{Le}\label{SubstituteLebesque}
Let~$\mu$ be a signed measure\textup, let~$A_+$ and~$A_-$ be the sets of its Hahn decomposition\textup, let~$\mu_+$ be its positive part. Consider the set
\begin{equation}\label{p_+}
P_+ = \{x \in A_+\mid \exists \, \delta(x) \quad \hbox{such that}\quad \forall  r < \delta(x)\quad \mu_+\big(B_r(x)\big) \leq 10\mu\big(B_r(x)\big)\}.
\end{equation}
Then~$\mu(P_+) = \mu(A_+)$.
\end{Le} 
\begin{proof}
Let~$Q_+$ be~$A_+ \setminus P_+$, we are going to show that~$\mu_+(Q_+) = 0$. Take any~$\eps > 0$, and let~$U_{\eps}$ be an open set such that~$A_+ \subset U_{\eps}$ and~$\mu_-(U_{\eps}) < \eps$. For every point~$x$ in~$Q_+$ there exists a sequence~$r_k$,~$r_k \to 0$, such that~$\mu_+(B_{r_k}(x)) \geq 10 \mu(B_{r_k}(x))$ and~$B_{r_k}(x) \subset U_{\eps}$. Therefore, by the Vitali covering theorem (see~\cite{M}, Theorem~$2.8$), there exists a disjoint collection of such balls~$B_{r_j}(x_j)$ such that~$\mu_+(Q_+\setminus \cup_{j}B_{r_j}(x_j)) = 0$. Therefore,~$\mu_+(\cup_{j}B_{r_j}(x_j)) \geq \mu_+(Q_+)$. On the other hand,~$\big|\mu_-(\cup_{j}B_{r_j}(x_j))\big| \leq \big|\mu_-(U_{\eps})\big| < \eps$. We can write
\begin{multline*}
\mu_+\big(\cup_{j}B_{r_j}(x_j)\big) = \sum\limits_{j}\mu_+\big(B_{r_j}(x_j)\big) \geq 10\sum\limits_{j}\mu\big(B_{r_j}(x_j)\big)=\\
10\sum\limits_{j}\mu_+\big(B_{r_j}(x_j)\big)+ 10\sum\limits_{j}\mu_-\big(B_{r_j}(x_j)\big) \geq 10\mu_+\big(\cup_{j}B_{r_j}(x_j)\big)-10\eps.
\end{multline*}
So,~$9\mu_+(\cup_{j}B_{r_j}(x_j)) \leq 10\eps$, consequently,~$9\mu_+(Q_+)\leq 10\eps$. Making~$\eps$ arbitrary small, we get~$\mu_+(Q_+) = 0$.
\end{proof}

Consider now the sets~$P_+^{(N)}$ given by the formula
\begin{equation*}
P_+^{(N)} = \Big\{x\in A_+\Big| \forall r < \frac{1}{N} \quad \mu_+\big(B_r(x)\big) \leq 10\mu\big(B_r(x)\big)\Big\}.
\end{equation*}
Surely,~$P_+ = \cup_{N}P_+^{(N)}$. Therefore, for every~$\eps > 0$ there exists~$N \in \mathbb{N}$ such that~$\mu_+\big(P_+^{(N)}\big) \geq \mu_+(A_+) - \eps$. We need to change inequality~\eqref{p_+} slightly.
\begin{Lemma}
Suppose that for some~$x$ fixed and all~$r \leq 2\delta$ the inequality~$\mu_+\big(B_r(x)\big) \leq 10\mu\big(B_r(x)\big)$ holds true. Then
\begin{equation}\label{p_+smoothed}
\int\varphi_r(x+y)\,d\mu_+(y) \leq 10 \int\varphi_r(x+y)\,d\mu(y)
\end{equation}
for all~$r < \delta$ and any radial non-negative test-function~$\varphi$ supported in~$B_1(0)$ that decreases as the radius grows.
\end{Lemma}
We leave the verification of this lemma to the reader. 
\begin{Lemma}\label{HahnDecompositionAndDimension}
Let~$\mu$ be a signed measure. Let~$\mu_+$ and~$\mu_-$ be its positive and negative parts. Then~$\dim\mu =\min(\dim\mu_+,\dim\mu_-)$.
\end{Lemma}
\begin{proof}
The inequality~$\dim\mu \geq \min(\dim\mu_+,\dim\mu_-)$ is obvious. 

Let~$\mu_+$ be concentrated on~$A_+$ and let~$\mu_-$ be concentrated on~$A_-$, i.e.~$\mu_{+}(A+) = \mu(A_+)$ and~$\mu_-(A_-) = \mu(A_-)$. Assume the contrary,~$\dim\mu > \dim\mu_+$. This means that there exists some Borel set~$F$ such that~$\mu_+(F) > 0$ and~$\dim F < \dim\mu$. Surely,~$\mu_+(A_+\cap F) > 0$ and~$\dim \big(A_+\cap F\big) < \dim\mu$. But~$\mu(A_+\cap F) = \mu_+(A_+\cap F) > 0$, a contradiction. 
\end{proof}

\paragraph{Proof of Lemma~\ref{FrostmanGeneralization}.}
We assume the contrary, suppose there exists some Borel set~$F$ such that~$\mu(F) \ne 0$, but~$\dim(F) < \alpha^- < \alpha$. By Lemma~\ref{HahnDecompositionAndDimension}, we may assume that~$F \subset A_+$, moreover, we may assume that~$F \in P_+^{(N)}$ for some big~$N$ (because these sets tend to~$A_+$ in measure) and~$F$ is compact (by the regularity of~$\mu$). Let~$\mu(F) = c_0$. By definition of the Hausdorff dimension, there exists a covering of~$F$ with the balls~$B_{r_j}(x_j)$ whose centers~$x_j$ lie in~$F$, whose radii~$r_j$ do not exceed~$\delta$ (which we take to be less than~$\frac{1}{10N}$), and~$\sum\limits_{j}r_j^{\alpha^-} \leq c_1$ for some uniform constant~$c_1$. We divide the set of balls into the classes of approximately equal balls:
\begin{equation*}
E_k = \big\{j\mid r_j \in (2^{-k-1},2^{-k}]\big\}.
\end{equation*}
Surely,~$|E_k| \leq 2^{k\alpha^-}c_1$. By the pigeonhole principle, there exists some~$k \gtrsim \log \frac{1}{\delta}$ such that
\begin{equation*}
\mu_+\big(F\cap\cup_{j\in E_k}B_{r_j}(x_j)\big) \geq \frac{c_0}{k^2}.
\end{equation*}
We fix~$\delta$ and also fix this~$k$ for a while. Let~$D_k$ be a subset of~$E_k$ such that~$\{x_j\mid j \in D_k\}$ is a maximal~$2^{-k}$-separated subset of~$\{x_j\mid j \in E_k\}$. Then
\begin{enumerate}
\item $\cup_{j \in D_k} B_{3r_j}(x_j) \supset F \cap \cup_{j \in E_k} B_j$, so~$\sum\limits_{j \in D_k} \mu_{+}\big(B_{3r_j}(x_j)\big) \geq \frac{c_0}{k^2}$;
\item The collection~$B_{4r_j}(x_j)$,~$j \in D_k$ covers each point only a finite number of times (uniformly).

Using these two statements and recallin that~$\varphi$ is equal~$1$ on~$B_{\frac34}(0)$, we can write
\begin{multline*}
\frac{c_0}{k^2} \leq \sum\limits_{j \in D_k} \mu_{+}(B_{3r_j}(x_j)) \leq \sum\limits_{j \in D_k} \int \varphi_{4r_j}(x_j + y)\,d\mu_+(y) \leq \\
10 \sum\limits_{j \in D_k} \int \Big|\varphi_{4r_j}(x_j + y)\,d\mu(y)\Big| \lesssim \Big(\sum\limits_{j \in D_k} r_j^{\alpha}\Big)^{\beta}\leq
\Big(|D_k|2^{-k\alpha}\Big)^{\beta} \lesssim c_1^{\beta}2^{\beta k (\alpha^- - \alpha)} \lesssim \delta^{\beta(\alpha - \alpha^-)}.
\end{multline*}
We get a contradiction for~$\delta$ small enough.
\end{enumerate}

\section{Examples and conjectures}\label{s4}
We note that Theorem~\ref{DifferentDerivativesTheorem} is sharp in the sense that one cannot rise the dimension. Consider a one-dimensional signed measure
\begin{equation}\label{FiniteDifference}
\Delta_h^s = \sum\limits_{j=0}^{s} (-1)^{s-j}C_s^j \delta_{hj}.
\end{equation}
This measure has~$s$ vanishing moments, therefore, there exists a compactly supported function~$f_h^s$ such that~$D^{s} f_h^s = \Delta_{h}^s$. Consider a function~$F$ on~$\mathbb{R}^n$ given by formula
\begin{equation*}
F(x) = \prod_{i=1}^n f_{h}^{m_i}(x_i).
\end{equation*}
Then, for each~$i$,~$D_i^{m_i} F$ is a measure supported on~$(n-1)$-dimensional hypercubes
\begin{equation*}
\{x \mid x_i = hj,\,\, \forall k \ne i \quad x_k \in [0,(m_k+1)h]\},
\end{equation*}
here~$j = 0,1,2,\ldots,m_i$. This proves that Theorem~\ref{DifferentDerivativesTheorem} is sharp.

Theorem~$4$ from~\cite{Kol} we have used is very strong. For the isotropic case, we what we need is the inequality~$\|D_i^{m-1} f\|_{L^{\frac{n}{n-1}}} \leq \|f\|_{W^m_1}$, which is much easier (see~\cite{S}). However, even embedding theorems from~\cite{S} are not enough strong for our purposes in the general setting (they need additional assumptions on the numbers~$m_i$).

We believe that the relations between Conjecture~\ref{GeneralConj} and embedding theorems are more deep. Maybe, embedding theorems for vector fields from~\cite{vS2} may help (there was a lot of progress in the recent years for the isotropic case, starting with~\cite{BB,vS}; see~\cite{KMS} for some examples of anisotropic theorems of such kind). 


{\small

D.M.~Stolyarov,\\
St.~Petersburg Department 
of Steklov Mathematical Institute RAS,\\Fontanka 27, St.~Petersburg, 
Russia;\\
Chebyshev Laboratory (SPbU), \\14th Line 29B, Vasilyevsky Island, St.~Petersburg, Russia.\\
e-mail: dms@pdmi.ras.ru.\\
\bigskip

M.~Wojciechowski,\\
Institute of Mathematics, 
Polish Academy of Sciences\\
00-956 Warszawa, 
Poland.\\
e-mail: miwoj-impan@o2.pl.
}
\end{document}